\documentclass[final,3p]{elsarticle}
\usepackage{tikz}
\usepackage{pgfplots}
\pgfplotsset{compat=1.18}
\usepackage{amssymb}
\usepackage{multirow} 
\usepackage{xcolor}
\usepackage{mathtools}
\usepackage{amsmath}
\usepackage[utf8]{inputenc}
\usepackage{mathabx}
\usepackage{amsfonts}
\usepackage{tabularx}
\usepackage{slashed}
\usepackage{dcolumn}
\usepackage{accents}
\usepackage{array}
\usepackage{float}
\usepackage{amsmath,amssymb,mathtools,amsthm}
\usepackage{graphicx,subcaption,booktabs}
\usepackage[hidelinks]{hyperref}
\usepackage[nameinlink,capitalise]{cleveref}
\usepackage{siunitx}
\usepackage{microtype}

\newtheorem{remark}{Remark}
\newtheorem{theorem}{Theorem}
\newtheorem{lemma}{Lemma}
\newtheorem{corollary}{Corollary}
\newcommand{\lL}{\mathrm{L}}
\newcommand{\lC}{\mathrm{C}}
\newcommand{\lR}{\mathrm{R}}
\newcommand{\fg}{\eta\tau}
\newcommand{\fgsq}{\eta^2\tau^2}

\newcommand{\st}{\sigma^2\tau}
\begin{document}

\begin{frontmatter}

\title{Nonuniform-Grid Markov Chain Approximation of \\ Continuous Processes with Time-Linear Moments}

 \author[label1]{Do Hyun Kim}
 \author[label2]{Ahmet Cetinkaya}
 \affiliation[label1]{organization={Shibaura Institute of Technology},
            addressline={am23105@shibaura-it.ac.jp},
             city={Tokyo},
             country={Japan}}
\affiliation[label2]{organization={Shibaura Institute of Technology},
             addressline={ahmet@shibaura-it.ac.jp},
             city={Tokyo},
             country={Japan}}



\begin{abstract}
We propose a method to approximate continuous-time, continuous-state stochastic processes by a discrete-time Markov chain defined on a nonuniform grid. Our method provides exact moment matching for processes whose first and second moments are linear functions of time. In particular, we show that, under certain conditions, the transition probabilities of a Markov chain can be chosen so that its first two moments match prescribed linear functions of time. These conditions depend on the grid points of the Markov chain and the coefficients of the linear mean and variance functions. Our proof relies on two recurrence relations for the expectation and variance across time. This approach enables simulation-based numerical analysis of continuous processes while preserving their key characteristics. We illustrate its efficacy by approximating continuous processes describing heat diffusion and geometric Brownian motion (GBM). For heat diffusion, we show that the heat profile at a set of points can be investigated by embedding those points inside the nonuniform grid of our Markov chain. For GBM, numerical simulations demonstrate that our approach, combined with suitable nonuniform grids, yields accurate approximations, with consistently small empirical Wasserstein-1 distances at long time horizons.
\end{abstract}




\end{frontmatter}

\section{Introduction}\label{sec1} 

Many real-world systems are described by continuous-time stochastic processes and stochastic differential equations. Prototypical examples include Brownian motion, the Ornstein–Uhlenbeck process, and Lévy processes, each with applications across physics, finance, and engineering. In computational physics, the fundamental solution of the heat equation coincides with the probability density function of Brownian motion, motivating particle simulations where particles follow Brownian paths to approximate heat diffusion. In finance, geometric Brownian motion is a standard model for asset prices in the Black–Scholes framework. However, practical applications often require numerical approximations of continuous-time, continuous-space processes, since analytical solutions are not always available and, even when they are, can be costly to evaluate or store. Discrete-time and discrete-space approximations are therefore desirable and, in particular, enable efficient particle methods allowing parallel processing \cite{chertock2001particle} and large-scale GPU implementations~\cite{solomon2010option}.

While there are discretization methods like finite element method and finite difference methods~\cite{thomas2013numerical} for approximations of engineering processes, there are not many methods of approximating continuous stochastic processes with mathematical guarantees. In this paper, we aim to construct a discrete-time Markov chain defined over a nonuniform grid composed of infinitely many points on one dimensional line so as to approximate continuous-time stochastic processes for which the mean and variance are linear functions of time. We show that under certain conditions, our aim can be achieved and the transition probabilities of the Markov chain can be assigned so as to exactly match the mean and the variance of the continuous-time process. These transition probabilities are obtained as functions of mean and variance coefficients of the process, as well as the distances between the grid points.

We apply our approach to approximating heat diffusion and geometric Brownian motion (GBM), two important examples of continuous-time stochastic processes. In both cases of approximation, we introduce a time-scaling factor so that our proposed Markov chain evolves on a discrete time index that is has linear relationship with the physical time of the continuous process. For heat diffusion in one dimension, we consider the temperature evolution at selected spatial locations and embed these locations into our nonuniform grid. The resulting Markov chain approximation enables simulation-based investigation of heat profiles without solving the underlying partial differential equation directly. For GBM, which models asset price dynamics in mathematical finance, we construct a Markov chain on a nonuniform grid whose mean and variance match those of the log-return of a given GBM. This results in an approximation whose empirical distribution remains close to that of GBM over long time horizons. In particular, our numerical experiments demonstrate that, with appropriately chosen grids and time-scaling, the empirical Wasserstein-1 distances between the Markov chain and the GBM distributions remain consistently small at long time horizons. 


The remainder of this paper is organized as follows. In Section~\ref{sec:related_works}, we provide an overview of related works. Then, in Section~\ref{sec:main}, we present our nonuniform grid Markov chain method and provide theoretical results. In Section~\ref{sec:application_heat}, we show how the heat diffusion process can be approximated with our method and its brief application. In Section~\ref{sec:application_GBM}, we demonstrate the application to approximating geometric Brownian motion (GBM) with numerical examples. Finally, in Section~\ref{sec:conclusion}, we conclude our paper and discuss our future works.

\section{Related Works}
\label{sec:related_works}
In this section, we provide comparisons of our work with existing literature on approximation of continuous processes, methods that use nonuniform grids, and lattice structures used in approximation. 


\subsection{Markov Chains for Approximation of Stochastic Processes}
Regarding the approximation of continuous-time stochastic processes, there are research that approximate different types of processes. Specifically, from the past literature, several researchers investigated the approximation of continuous-time stochastic processes by using discrete-time Markov chains. In particular, \cite{lefebvre2009first} analyzed a discrete version of the Ornstein-Uhlenbeck process by modeling it as a Markov chain, and derived expressions for the probability that the Markov chain first reaches one boundary before the other (i.e., first hitting place probability). Building on this approach, \cite{lefebvre2011first} further investigated the use of a Markov chain on a uniform grid that converges to a GBM when the spacing between the grid points and the interval between discrete-time steps approach zero. In \cite{lefebvre2011first}, the transition probabilities of the Markov chain are designed to reflect the drift and volatility of a GBM. Our approach is different from that of \cite{lefebvre2011first} in a few aspects.  First of all, while \cite{lefebvre2011first} focused on the convergence of a finite-state Markov chain to the GBM to analyze first hitting probabilities, our objective is the approximation of the log-return of GBM using Markov chains with \emph{infinite} state spaces.  We also note that \cite{lefebvre2011first} chose uniform grids for their analysis, whereas we consider nonuniform grids. Previously, \cite{frannek2012one} and \cite{frannek2014stochastic} used continuous-time Markov chains on nonuniform grids for approximation purposes; however, their goal was to approximate the fundamental solution of the heat equation. Our approach is targeted to develop more generalized discrete-time Markov chain framework that can be calibrated to continuous-time processes whose mean and variance grow linearly in time.

\subsection{Methods Using Nonuniform Grids}
Nonuniform grids have been used in finite-difference methods for approximating solutions of partial differential equations (see, e.g. \cite{thomas2013numerical}). The structure of a nonuniform grid offers several advantages, including computational efficiency and localized resolution refinement~\cite{jianchun1995high}. For instance, if one is interested in precise solutions in a certain region of the domain, the grid may be set to include many points in that region and fewer points in regions where precision is not required. Owing to these benefits, previous studies have employed nonuniform grids in their analyses. For example, Bodeau et al. used nonuniform grids to solve partial differential equations in finance~\cite{bodeau2000non}. Our paper takes advantage of the localized resolution refinement property of a nonuniform grid. In particular, we demonstrate that grid points can be selected to obtain approximations better than those obtained with a uniform grid with numerical examples in Section \ref{sec:application_GBM}.  

\subsection{Lattice Structures Used in Approximation}
Our proposed Markov chain that we use for approximation has a three-branch local structure. From any given state, the Markov chain can move to a lower state, stay at the same state, or move to a higher state. This is closely related to the trinomial lattice structure widely used in asset pricing. In quantitative finance, binomial and trinomial trees have long been employed to approximate GBM for option pricing. In binomial models, the asset price moves up or down at each time step, while trinomial models extend this by also allowing the price to remain unchanged. Classical examples include the Cox–Ross–Rubinstein model~\cite{cox1979option}, the Jarrow–Rudd model~\cite{jarrow1983option}, and the Tian tree model~\cite{tian1993modified}. More recent developments, such as trinomial Markov tree models with recombining nodes~\cite{xiaoping2014pricing}, offer faster convergence, higher accuracy, and improved computational efficiency. Furthermore, \cite{kim2016multi} proposes a unified binomial framework that matches all moments of GBM over finite time intervals. However, to the best of our knowledge, existing lattice-based approaches do not exploit nonuniform grids to improve empirical approximation quality.





\section{Markov Chains on Nonuniform Grids with Time-Linear First Two Moments}\label{sec:main}
In this section, we first provide an overview of our notation, and then we present our main technical result that allows us to choose the transition probabilities of a nonuniformly-gridded Markov chain so that its mean and variance match given linear functions of time.

\subsection{Notation and Preliminaries}
We use $\mathbb{Z}$, $\mathbb{N}_0$, and $\mathbb{N}$ to denote the lists of all integers, nonnegative integers, and positive integers, respectively. In this paper, we define stochastic processes on a probability space with $\mathbb{P}$ denoting the probability measure, $\mathbb{E}[\cdot]$ denoting the expectation and $\mathrm{Var}[\cdot]$ denoting the variance.

\subsection{Markov Chain and Its Characterization}
First, we define a discrete-time Markov chain $\{r(k)\in \mathcal{X}\}_{k\in \mathbb N_0}$ on a nonuniform grid given by 
\begin{align}
    \mathcal{X}\coloneq\{x_i\in\mathbb R: i\in \mathbb{Z}\},
\end{align}
where $x_i\in \mathbb{R}$ are grid points that satisfy $x_i<x_{i+1}$ for $i\in \mathbb{Z}$ and $x_0=0$. 

Transition probabilities of the Markov chain $\{r(k)\}_{k\in \mathbb N_0}$ are  characterized as 
\begin{align}
    \mathbb{P}(r(k+1)=x_{j}\mid r(k)=x_i) = \begin{cases}
    \lambda_{i,\lL} , &\text{if}\,\, j=i-1, \\
    \lambda_{i,\lC} , &\text{if}\,\, j=i,   \\
    \lambda_{i,\lR} , &\text{if}\,\,j=i+1, \\
    0   , &\,\, \mathrm{otherwise},
    \end{cases}\label{Markov_lambda} 
\end{align}
and its initial distribution is characterized as $\mathbb{P}(r(0)=x_i)=\nu_i$, $i\in \mathbb{Z}$.

We are ready to state our main result, which provides transition probabilities of the Markov chain.

\begin{theorem} \label{thrm:markov}
Given $M,V\in \mathbb{R}$, if the inequalities
\begin{align}
     &M(x_{i+1}-x_i) \leq M^2 +V, \label{inequality}\\
     &-M(x_i-x_{i-1}) \leq M^2 +V, \label{inequality_1} \\
    &M^2+V+M(2x_i-x_{i+1}-x_{i-1}) \leq (x_{i+1}-x_i)(x_i-x_{i-1})\label{inequality_2} 
\end{align}
hold for each $i\in \mathbb{Z}$, then $\{r(k)\in\mathcal{X}\}_{k\in \mathbb N_0}$ with transition probabilities in~\eqref{Markov_lambda} where
    \begin{align}
     &\lambda_{i,\lL} = \frac{M^2+V-M(x_{i+1}-x_i)}{(x_{i+1}-x_{i-1})(x_i-x_{i-1})},\label{lambdaL}\\
     &\lambda_{i,\lR} = \frac{M^2+V+M(x_i-x_{i-1})}{(x_{i+1}-x_{i-1})(x_{i+1}-x_{i})},\label{lambdaR}\\
     &\lambda_{i,\lC} = 1-\frac{M^2+V+M(2x_i-x_{i+1}-x_{i-1})}{(x_{i+1}-x_i)(x_i-x_{i-1})},\label{lambdaC}
\end{align} and initial distribution $\nu_0=1$, $\nu_i=0$ for $i\neq0$, is a well-defined Markov chain, and satisfies 
\begin{align}
    \mathbb{E}[r(k)] = Mk,\quad
    \mathrm{Var}[r(k)] = Vk \label{thrm:exp_var}
\end{align}
for every $k\in\mathbb{N}_0$.
\end{theorem}

Theorem~\ref{thrm:markov} allows us to choose the transition probabilities $(\lambda_{i,\lL}, \lambda_{i,\lC}, \lambda_{i,\lR})$ of a discrete-time Markov chain so that its expectation and variance match  given linear functions of the discrete time index (i.e., $Mk$ and $Vk$). In this paper our goal is to use Theorem~\ref{thrm:markov} for approximating continuous-time processes with a physical time variable $t\in [0,\infty)$. To match the discrete time index $k$ with this physical time, we will consider a time-scaling factor in Sections~\ref{sec:application_heat} and \ref{sec:application_GBM}.

\begin{remark} 
Conditions \eqref{inequality}--\eqref{inequality_2} ensure that $V \geq 0$. To see this note that \eqref{inequality} and \eqref{inequality_2} imply $M(x_{i+1}-x_i)+M(2x_i-x_{i+1}-x_{i-1})\leq(x_{i+1}-x_i)(x_i-x_{i-1})$, which implies $M(x_i-x_{i-1})\leq(x_{i+1}-x_i)(x_i-x_{i-1})$ or equivalently 
\begin{align}
    M \leq (x_{i+1}-x_i). \label{ineql:R}
\end{align}
Furthermore, \eqref{inequality_1} and \eqref{inequality_2} imply $-M(x_{i+1}-x_i)\leq (x_{i+1}-x_i)(x_i-x_{i-1})$ or equivalently
\begin{align}
    -M \leq (x_i-x_{i-1}). \label{ineql:L}
\end{align}
Using \eqref{ineql:R} and \eqref{ineql:L}, we have if $M>0$, from \eqref{inequality}, $M\{(x_{i+1}-x_i)-M\}\leq V$ implies $V\geq0$. If $M\leq0$, from \eqref{inequality_1}, $-M\{(x_i-x_{i-1})+M\} \leq V$ leads to $V\geq0$.
\end{remark}
\begin{remark}
    One notable case in Theorem~\ref{thrm:markov} arises when the grid is uniform, meaning that the distance between adjacent states is constant, that is,
\begin{align*}
    x_{i+1}-x_i = h\,\,\text{for all}\,\,i \in \mathbb{Z},
\end{align*}
where $h>0$.
In this case, we have
 \begin{align*}
     \lambda_{i,\lL} = \frac{1}{2h^2}\left(M^2+V- h M\right), \quad \lambda_{i,\lR} = \frac{1}{2h^2}\left(M^2+V+ h M\right),\quad
     \lambda_{i,\lC} = 1-\frac{1}{h^2}\left(M^2+V\right).
\end{align*}
The transition probabilities do not depend on the index $i$ anymore, and they are constant. This leads to a decrease in the complexity of computation in simulations.
\end{remark}

\subsection{Proof of Theorem~\ref{thrm:markov}}
Proof of Theorem~\ref{thrm:markov} relies on the following lemma.
\begin{lemma}\label{lemma_1}
    For any $f: \mathcal{X} \to \mathbb{R}$, we have
    \begin{align}
        \sum_{i = -\infty}^{\infty}f(x_i)\mathbb{P}(r(k)=x_i)
        =\sum_{i = -\infty}^{\infty}(f(x_{i-1})\lambda_{i,L}+f(x_{i})\lambda_{i,C}+f(x_{i+1})\lambda_{i,R})\mathbb{P}(r(k-1)=x_i).\label{lemma_eq}
    \end{align}
\begin{proof}
By \eqref{Markov_lambda}, $\mathbb{P}(r(k)=x_i\mid r(k-1)=x_j)=0$ for $j\notin \{-1, 0, 1\}$. Thus, by Law of total probability, we obtain 
\begin{align}
    \mathbb{P}(r(k)=x_i)=\sum_{j=-1}^{1}\mathbb{P}(r(k)=x_i\mid r(k-1)=x_j)\mathbb{P}(r(k-1)=x_j),
\end{align}
which implies
\begin{align}
\sum_{i = -\infty}^{\infty} f(x_i)\mathbb{P}(r(k) = x_i)
&= \sum_{i = -\infty}^{\infty} f(x_i)\mathbb{P}(r(k) = x_{i} \mid r(k-1) = x_{i+1}) \mathbb{P}(r(k-1) = x_{i+1})\nonumber\\
&\quad+ \sum_{i = -\infty}^{\infty} f(x_i)\mathbb{P}(r(k) = x_{i} \mid r(k-1) = x_i) \mathbb{P}(r(k-1) = x_{i})\nonumber\\ 
&\quad+\sum_{i = -\infty}^{\infty} f(x_i)\mathbb{P}(r(k) = x_{i} \mid r(k-1) = x_{i-1}) \mathbb{P}(r(k-1) = x_{i-1}). \label{three-sums}
\end{align}
Since the ranges of the summations on the right-hand side of \eqref{three-sums} are from $-\infty$ to $+\infty$, the indices $i$ can be reparameterized. This results in
\begin{align}
\sum_{i = -\infty}^{\infty} f(x_i)\mathbb{P}(r(k) = x_i) &= \sum_{i = -\infty}^{\infty} f(x_{i-1})\mathbb{P}(r(k) = x_{i-1} \mid r(k-1) = x_i) \mathbb{P}(r(k-1) = x_{i})\nonumber\\
&\quad+ \sum_{i = -\infty}^{\infty} f(x_i)\mathbb{P}(r(k) = x_{i} \mid r(k-1) = x_i) \mathbb{P}(r(k-1) = x_{i})\nonumber\\ 
&\quad+\sum_{i = -\infty}^{\infty} f(x_{i+1})\mathbb{P}(r(k) = x_{i+1} \mid r(k-1) = x_i) \mathbb{P}(r(k-1) = x_{i}).
\end{align}
By~\eqref{Markov_lambda}, we can replace conditional probabilities with $\lambda_{i,\lL}, \lambda_{i,\lC}, \lambda_{i,\lR}$ to get
\begin{align}
\sum_{i = -\infty}^{\infty} f(x_i)\mathbb{P}(r(k) = x_i) &= \sum_{i = -\infty}^{\infty} f(x_{i-1})\lambda_{i,\lL} \mathbb{P}(r(k-1) = x_{i})
+ \sum_{i = -\infty}^{\infty} f(x_i)\lambda_{i,\lC}\mathbb{P}(r(k-1) = x_{i})\nonumber\\
&\quad\quad+\sum_{i = -\infty}^{\infty} f(x_{i+1})\lambda_{i,\lR}\mathbb{P}(r(k-1) = x_{i}),
\end{align}
which implies~\eqref{lemma_eq}.
\end{proof}
\end{lemma}

\vskip 10pt

\begin{proof}[Proof of Theorem~\ref{thrm:markov}] First of all, by using \eqref{inequality}--\eqref{inequality_2}, we obtain
\begin{align}
     &0\leq \frac{M^2+V-M(x_{i+1}-x_i)}{(x_{i+1}-x_{i-1})(x_i-x_{i-1})},\,\label{f-inequality}\\
     &0\leq \frac{M^2+V+M(x_i-x_{i-1})}{(x_{i+1}-x_{i-1})(x_{i+1}-x_{i})},\label{f-inequality1}\\
     &\frac{M^2+V+M(2x_i-x_{i+1}-x_{i-1})}{(x_{i+1}-x_i)(x_i-x_{i-1})}\leq 1. \label{f-inequality2}
\end{align} 
Notice that \eqref{lambdaL}--\eqref{lambdaC} together with \eqref{f-inequality}--\eqref{f-inequality2} imply
\begin{align}
    0 \leq\lambda_{i,L}\leq1,\,\, 0 \leq\lambda_{i,C}\leq1,\,\, 0\leq\lambda_{i,R} \leq1. \label{lambda-ineqs}
\end{align}
Furthermore, $\lambda_{i,\lL}+\lambda_{i,\lR}+\lambda_{i,\lC}=1$ holds by \eqref{lambdaL}--\eqref{lambdaC}. Thus, the Markov chain $\{r(k)\}_{k\in \mathbb{N}_0}$ is well defined. Now, by the definition of expectation, we get
\begin{align}
         \mathbb{E}[r(k)] &= \sum_{i =-\infty}^{\infty}x_i\mathbb{P}(r(k)=x_i).
\end{align}
Here, by using~Lemma~\ref{lemma_1} with $f(x_i)=x_i$, we obtain
\begin{align}
    \sum_{i =-\infty}^{\infty}(x_{i-1}\lambda_{i,L}+x_{i}\lambda_{i,C}+x_{i+1}\lambda_{i,R})\mathbb{P}(r(k-1)=x_i).
\end{align}
Replacing $\lambda_{i,C}$ with $1-\lambda_{i,L}-\lambda_{i,R}$ leads to 
    \begin{align}
         \mathbb{E}[r(k)] =\sum_{i =-\infty}^{\infty}\left\{(x_{i-1}-x_i)\lambda_{i,L}+(x_{i+1}-x_{i})\lambda_{i,R}+x_i\right\}\mathbb{P}(r(k-1)=x_i).\label{expt_last}
    \end{align}
Substituting the expressions for $\lambda_{i,L}$ and $\lambda_{i,R}$ from \eqref{lambdaL} and \eqref{lambdaR}, we arrive at
\begin{align}
    \mathbb{E}[r(k)] &= \sum_{i =-\infty}^{\infty}\left\{x_i+\frac{(x_{i+1}-x_i)M+(x_i-x_{i-1})M}{(x_{i+1}-x_{i-1})}\right\}\mathbb{P}(r(k-1)=x_i) \nonumber \\
    &= \sum_{i =-\infty}^{\infty}(x_i+M)\mathbb{P}(r(k-1)=x_i)\nonumber\\
    &= \sum_{i =-\infty}^{\infty}x_i\mathbb{P}(r(k-1)=x_i)+M\sum_{i =-\infty}^{\infty}\mathbb{P}(r(k-1)=x_i)\nonumber\\
    &= \mathbb{E}[(r(k-1))]+M.\label{expt_before}
\end{align}
By using mathematical induction on \eqref{expt_before} and $r(0)=0$, we conclude
\begin{align}
    \mathbb{E}[r(k)] = \mathbb{E}[r(0)] + M \cdot k
    = M k\label{expt_final},
\end{align}
which implies~\eqref{thrm:exp_var}.

We now turn to the variance. First, we evaluate
\begin{align}
     \mathbb{E}[r^2(k)] &=\sum_{i =-\infty}^{\infty}x_i^2\mathbb{P}(r(k)=x_i).\label{def_ex^2}
\end{align}
Using Lemma~\ref{lemma_1} with $f(x_i)=x_i^2$, we transform equation~\eqref{def_ex^2} into
\begin{align}
    \mathbb{E}[r^2(k)] = \sum_{i =-\infty}^{\infty}(x^2_{i-1}\lambda_{i,L}+x^2_{i+1}\lambda_{i,R}+x^2_{i}\lambda_{i,C})\mathbb{P}(r(k-1)=x_i). \label{after-second-use-of-lemma1}
\end{align}
Since, $\lambda_{i,C}=1-\lambda_{i,L}-\lambda_{i,R}$, it follows from \eqref{after-second-use-of-lemma1} that
\begin{align*}
    \mathbb{E}[r^2(k)]=\sum_{i =-\infty}^{\infty}\left\{(x_{i-1}^2-x_i^2)\lambda_{i,L}+(x_{i+1}^2-x_{i}^2)\lambda_{i,R}+x_i^2\right\}\mathbb{P}(r(k-1)=x_i). 
\end{align*}
Now, we use the expressions for $\lambda_{i,L}$ and $\lambda_{i,R}$ in this equation to obtain
\begin{align}
    \mathbb{E}[r^2(k)]&=\sum_{i =-\infty}^{\infty}\left\{\frac{(x_{i+1}-x_{i-1})(M^2+V+2x_iM)}{(x_{i+1}-x_{i-1})}+x_i^2\right\}\mathbb{P}(r(k-1)=x_i)\nonumber \\
    &=\sum_{i =-\infty}^{\infty}\left(M^2+V+2x_i M+x_i^2\right)\mathbb{P}(r(k-1)=x_i) \nonumber \\
    &=M^2+V+2M\sum_{i =-\infty}^{\infty}x_i\mathbb{P}(r(k-1)=x_i)+\sum_{i =-\infty}^{\infty}x_i^2\mathbb{P}(r(k-1)=x_i)\nonumber\\
     &=M^2+V+2M\cdot\mathbb{E}[r(k-1)]+\mathbb{E}[r^2(k-1)]. \label{exp-r2-relation}
\end{align}
Since $\mathbb{E}[r(k-1)]=M(k-1)$ by~\eqref{expt_final}, it follows from \eqref{exp-r2-relation} that
\begin{align}
    \mathbb{E}[r^2(k)] =M^2+V+2M^2(k-1)+\mathbb{E}[r^2(k-1)].\label{expt^2_last2}
\end{align}
This is a recurrence relation. Thus, by induction, we get
\begin{align}
    \mathbb{E}[r^2(k)] = \left(M^2+V\right)\cdot k+2\frac{M^2(k-1)}{2}\cdot k
    = (M)^2 k^2+V k . \label{expt^2_final}
\end{align}
Finally, by using the identity $\mathrm{Var}[r(k)] =\mathbb{E}[r^2(k)] - \mathbb{E}[r(k)]^2$, as well as with~\eqref{expt_final} and~\eqref{expt^2_final}, we calculate
\begin{align}
    \mathrm{Var}[r(k)] &=\mathbb{E}[r^2(k)] - \mathbb{E}[r(k)]^2 =M^2 k^2+Vk- M^2 k^2= Vk \label{Var_final}
\end{align}
confirming~\eqref{thrm:exp_var}.
\end{proof}
\section{Application to Approximating Stochastic Heat Diffusion}\label{sec:application_heat}
In this section, we demonstrate how the proposed Markov chain framework can be used to approximate stochastic heat diffusion in one spatial dimension.

\subsection{Approximation}
Among the continuous processes, heat diffusion has been considered significant, as it is used in wide range of science and engineering fields \cite{stocker1989design}. Heat diffusion benefits substantially from discretization because the analytical intractability of its governing partial differential equation makes a discrete-time approximation particularly advantageous for numerical analysis and simulation.

Heat diffusion is modeled by the equation
\begin{align}
    \frac{\partial u(t,x)}{\partial t} = \alpha \Delta u(t,x), 
\end{align}
where $u(t,x)$ represents the temperature at time $t$ at position $x$, $\alpha > 0$ represents the diffusivity constant, and $\Delta$ represents the Laplace operator. In this paper, we consider the one-dimensional case with $x\in \mathbb{R}$, which results in the one dimensional equation
\begin{align}
    \frac{\partial u}{\partial t} = \alpha \frac{\partial^2 u}{\partial x^2}.
\end{align} 
The solution to this heat equation is given by
\begin{align}
    u(t,x) = \frac{1}{(4\pi\alpha t)^{\frac{1}{2}}}e^{-\frac{x^2}{4\alpha t}}. \label{pdf-heat}
\end{align}
Notice that this solution is also the probability density function of Normal distribution with mean $0$ and variance $2\alpha t$ \cite{frannek2012one}. Stochastic heat diffusion considers heat particles at a certain time $t$ as realizations of a random variable with this distribution. Relation between heat equation and Normal distribution also enabled  researchers to establish a connection between random walks and heat diffusion (see, e.g., \cite{lawler2010random,frannek2012one} and the references therein), where random movement of heat particles are used for approximating the solution to heat diffusion.

We use a Markov chain defined on a nonuniform-grid to approximate the movement of these heat particles. Our approach is the discrete-time analogue of the continuous-time approach presented in \cite{frannek2012one}.
Specifically, in the following result, we provide transition probabilities of a discrete-time Markov chain that exactly matches the mean and variance of a heat particle whose probability distribution at time $t=k\tau$ is characterized by the probability density function in \eqref{pdf-heat}. Here we introduce a time-scaling factor $\tau$. The physical time of the continuous process is a linear function of the discrete time index of the Markov chain that we introduce (i.e., $t=k\tau$). 


\begin{corollary} \label{corollary:Heat} If the inequality
\begin{align}
    2\alpha\tau \leq (x_{i+1}-x_i)(x_i-x_{i-1}) \label{inequality_H}
\end{align}
holds for each $i\in \mathbb{Z}$, then $\{r(k)\in \mathcal{X}\}_{k\in \mathbb N_0}$ with transition probabilities in~\eqref{Markov_lambda} where
    \begin{align}
     &\lambda_{i,\lL} = \frac{2\alpha\tau}{(x_{i+1}-x_{i-1})(x_i-x_{i-1})},\label{corollary_H:lambdaL}\\
     &\lambda_{i,\lR} = \frac{2\alpha\tau}{(x_{i+1}-x_{i-1})(x_{i+1}-x_{i})},\label{corollary_H:lambdaR}\\
     &\lambda_{i,\lC} = 1-\frac{2\alpha\tau}{(x_{i+1}-x_i)(x_i-x_{i-1})},\label{corollary_H:lambdaC}
\end{align} and initial distribution $\nu_0=1$, $\nu_i=0$ for $i\neq0$, is a well-defined Markov chain, and satisfies 
\begin{align}
    \mathbb{E}[r(k)] = 0,\quad
    \mathrm{Var}[r(k)] = 2\alpha k\tau,
\end{align}
for every $k\in\mathbb{N}_0$.
\end{corollary}
\begin{proof}[Proof of Corollary \ref{corollary:Heat}]
This result is a consequence of Theorem\ref{thrm:markov} with $M=0$, $V=2\alpha\tau$, because \eqref{inequality_H} implies \eqref{inequality}, \eqref{inequality_1}, \eqref{inequality_2}, as we have $\alpha>0$ and $\tau >0$.
\end{proof}

Since the heat diffusion has mean 0, the inequalities~\eqref{corollary:Heat} are simplified compared to Theorem \ref{thrm:markov} and Corollary \ref{corollary:GBM}. To match the first two moments of heat diffusion, scaling factor $\tau$ and grid points should be appropriately chosen so that it satisfies the inequality in \eqref{inequality_H} considering the given diffusivity constant $\alpha$. 

\subsection{Markov Chain Approximation of Temperature Evolution}
As an application of our results, we are interested in how the temperature evolves around a certain set of points in one dimensional space
\begin{align}
    p_1, p_2, \ldots, p_m \in \mathbb{R},
\end{align} where $m \in \mathbb{N}$.  We want to know the temperature values at times $0, \tau, 2\tau, \ldots, n\tau$ with $n \in \mathbb{N}$. To this end, we construct a spatial grid 
\begin{align}
    \Tilde{\mathcal{X}} \triangleq \{x_{-n}, \cdots, x_{-1}, x_0, x_{1}, \cdots, x_n \},
\end{align} such that the set of points of interest satisfies $\mathcal{P} \triangleq \{p_1, p_2, \ldots, p_m \}\subset \Tilde{\mathcal{X}}$.

To approximate the heat, we can consider a finite-state Markov chain $\{r(k)\in \Tilde{\mathcal{X}}\}$ with initial distribution given by
\begin{align}
    \Tilde{\nu} = [\underbrace{0 \ldots 0}_{\text{n terms}} 1 \underbrace{0 \ldots 0}_{\text{n terms}}],
\end{align} 
and the transition probability matrix given as
\[
\Tilde{P} = 
\begin{pmatrix}
1 & 0 & 0 & \cdots & 0 \\[4pt]
\lambda_{-(n-1),L} & \lambda_{-(n-1),C} & \lambda_{-(n-1),R} & \ddots & \vdots \\[4pt]
0 & \ddots & \ddots & \ddots & 0 \\[4pt]
\vdots & \ddots & \lambda_{n-1,L} & \lambda_{n-1,C} & \lambda_{n-1,R} \\[4pt]
0 & \cdots  & 0 & 0 & 1
\end{pmatrix}
\]
so that 
\begin{align}
(\Tilde{\nu} \Tilde{P}^k)_{i+n+1} = \mathbb{P}[r(k)=x_i],\quad k \in \{0,1, \dots ,n\}, \quad i\in \{-n, \ldots, 0, \ldots, n\},   
\end{align}
where $(\Tilde{\nu} \Tilde{P}^k)_{i+n+1}$ corresponds to $(i+n+1)$th entry of vector $\Tilde{\nu} \Tilde{P}^k$.  We use Corollary~\ref{corollary:Heat} to choose $\lambda_{i,\lL}, \lambda_{i,\lC}, \lambda_{i,\lR}$. Corollary~\ref{corollary:Heat} is for infinite state Markov chains for infinite durations. In this practical example, we consider finite time steps. Therefore, we consider a truncated version of the Markov chain that goes from $x_{-n}$ to $x_{n}$. Notice that within $n$ time steps, starting from $x_0$, the furthest values that the Markov chain can reach are $x_{-n}$ and $x_{n}$.


This construction provides a discrete-time approximation of heat diffusion through a Markov chain defined on the spatial grid $\Tilde{\mathcal{X}}$. The initial distribution represents the origin of the concentrated heat, and transition probability matrix $\Tilde{P}$ encodes how heat spreads to neighboring points over each time increment $\tau$. Iteration of $\Tilde{P}$ mimics the diffusive spreading mechanism of the heat equation. As a result, the distribution of $\Tilde{\nu} \Tilde{P}^k$ represents  an approximated temperature profile at time $k\tau$, and the temperature at the points of interest $p_j \in \mathcal{P}$ can be observed with the corresponding entries of this vector. Furthermore, in terms of the spatial grid, as $n$ increases by making the grid finer, the approximation gets more precise.

\section{Application to Approximating Geometric Brownian Motion}\label{sec:application_GBM}

In this section, we present the approximation of log-return of geometric Brownian motion as an application of our approximation method in Theorem\ref{thrm:markov}. The efficacy of our approximation method and the use of a nonuniform grid will be illustrated through a numerical simulation.

\subsection{Application}
Geometric Brownian motion (GBM) is extensively used in finance, particularly for modeling stock prices, pricing derivatives, and other assets~\cite{farida2018stock,lefebvre2007applied}. As such, GBM has been a fundamental concept in financial mathematics. The probability distribution of GBM is described by the Black--Scholes partial differential equation, and although several numerical methods exist for approximating this equation's solution (see, e.g., \cite{cen2011robust,heider2010numerical,zadeh2025interval}), directly approximating GBM by a discrete stochastic model offers unique benefits, particularly in leveraging GPU for parallel computation in simulations. We believe that our approximation method can be useful in applications in finance, physics, engineering, and mathematical biology where geometric Brownian motion is used as a modeling tool (see, e.g., \cite{fattahi2022modeling,giordano2023infinite,lee2022geometric,di2018modeling,mishura2021discrete}).

Using our Markov chain model, we approximate the log-return of GBM (i.e., $\ln(s_{k\tau })$) with $r(k)$  by matching their first two moments, expectation and variance. The time-scaling factor, $\tau$ is introduced to allow approximation not only on discrete time steps, but at any given time $k\tau $.
 Geometric Brownian motion with drift and volatility coefficients $\mu$ and $\sigma^2$ is denoted with the continuous-time stochastic process $\{s_t\in\mathbb{R}\}_{t\geq0}$ that satisfies $\mathrm{d}s_t = \mu s_t\mathrm{d}t+\sigma s_t\mathrm{d}w_t$, where $s_0>0$ is a fixed constant and $\{w_t\in \mathbb{R}\}_{t\geq0}$ is the Wiener process~\cite{krylov2002introduction}. Note that $\mathbb{E}\left[\ln\left(s_{t}/s_0\right)\right] = \left(\mu-\sigma^2/2\right)t$ and
    $\mathrm{Var}\left[\ln\left(s_{t}/s_0\right)\right] = \sigma^2t$.
Therefore, if we incorporate a time-scaling factor $\tau>0$, for every $k\in \mathbb{N}_0$, we get
\begin{align}
    \mathbb{E}\left[\ln\left(\frac{s_{k\tau}}{s_0}\right)\right] = \left(\mu-\frac{\sigma^2}{2}\right)k\tau ,\quad
    \mathrm{Var}\left[\ln\left(\frac{s_{k\tau}}{s_0}\right)\right] = \sigma^2 k\tau .\label{expec-var-scale}
\end{align}

\begin{corollary} \label{corollary:GBM}
Let $\eta \coloneq\mu-\frac{\sigma^2}{2}$ and $\tau>0$. If the inequalities
\begin{align}
     &\fg(x_{i+1}-x_i) \leq \fgsq +\st, \label{corollary:ineql}\\
     &-\fg(x_i-x_{i-1}) \leq \fgsq +\st, \label{corollary:ineql_1}\\
    &\fgsq+\st+\fg(2x_i-x_{i+1}-x_{i-1}) \leq (x_{i+1}-x_i)(x_i-x_{i-1})\label{corollary:ineql_2}
\end{align}
hold for each $i\in \mathbb{Z}$, then $\{r(k)\in \mathcal{X}\}_{k\in \mathbb N_0}$ with transition probabilities in~\eqref{Markov_lambda} where
    \begin{align}
     &\lambda_{i,\lL} = \frac{\fgsq+\st-\fg(x_{i+1}-x_i)}{(x_{i+1}-x_{i-1})(x_i-x_{i-1})},\label{corollary:lambdaL}\\
     &\lambda_{i,\lR} = \frac{\fgsq+\st+\fg(x_i-x_{i-1})}{(x_{i+1}-x_{i-1})(x_{i+1}-x_{i})},\label{corollary:lambdaR}\\
     &\lambda_{i,\lC} = 1-\frac{\fgsq+\st+\fg(2x_i-x_{i+1}-x_{i-1})}{(x_{i+1}-x_i)(x_i-x_{i-1})},\label{corollary:lambdaC}
\end{align} and initial distribution $\nu_0=1$, $\nu_i=0$ for $i\neq0$, is a well-defined Markov chain, and satisfies 
\begin{align}
    \mathbb{E}[r(k)] = \mathbb{E}\left[\ln\left(\frac{s_{k\tau}}{s_0}\right)\right],\quad
    \mathrm{Var}[r(k)] = \mathrm{Var}\left[\ln\left(\frac{s_{k\tau}}{s_0}\right)\right],\label{thrm1_exp_var}
\end{align}
for every $k\in\mathbb{N}_0$.
\end{corollary}
\begin{proof}
This result is a consequence of Theorem\ref{thrm:markov} with $M=\eta\tau$ and $V=\sigma^2\tau$.
\end{proof}

\begin{remark}
Notice that for each state $x_i$ of the Markov chain in Corollary~\ref{corollary:GBM}, the next possible states are $x_{i-1}, x_i,x_{i+1}$. This structure is similar to the trinomial lattice structure often used in asset pricing models. We remark that approximation of GBM using binomial and trinomial models have been investigated in finance field for option pricing. In binomial models, the evolution of the price of an asset is modeled to move up or down between the points of a grid. As an extension of the binomial models, trinomial models also allow prices to remain constant. Existing binomial/trinomial models include the traditional ones such as the Cox-Ross-Rubinstein model \cite{cox1979option}, the Jarrow-Rudd model  \cite{jarrow1983option}, and the Tian tree model \cite{tian1993modified}, as well as the more recent trinomial Markov tree model with recombining nodes \cite{xiaoping2014pricing}, which allows faster convergence, higher accuracy, and efficient computation. Recently, \cite{kim2016multi} provided a unified framework of binomial models that matches all moments of GBM in finite time intervals. However, to the best of our knowledge utilization of nonuniform grids for better empirical performance have not been explored in the literature. 
\end{remark}

\subsection{Numerical Example}
In this section, we consider the problem of approximating log-return of GBM with coefficients $\mu=2$ and $\sigma^2=0.25$. Following the method that we presented in the previous section, we use a Markov chain defined on grid points 
\begin{align}
    x_i=\begin{cases}
        ci, \,\, &\text{if}\,\, i \geq 0, \\
        10ci, \,\, &\text{if}\,\, i < 0, \\
    \end{cases} \label{simul_grid}
\end{align}
for $i\in \mathbb{Z}$, where $c=0.01$ is fixed constant. Notice that the grid points on the right-side of the real line are set to have a finer resolution. This is to capture the behavior of GBM more closely, as the log-return tends to be positive with the selected drift coefficient $\mu=2$.

Note that the time-scaling factor $\tau=0.0002$ satisfies the inequalities~\eqref{corollary:ineql}--\eqref{corollary:ineql_2}. Therefore, by Corollary~\ref{corollary:GBM}, transition probabilities 
in \eqref{corollary:lambdaL}--\eqref{corollary:lambdaC} guarantee that the first two moments of $r(k)$ and $\ln(s_{k\tau}/s_0)$ match exactly, as in \eqref{thrm1_exp_var}. Using the transition probabilities derived in \eqref{corollary:lambdaL}--\eqref{corollary:lambdaC}, we generate $N=10000$ realizations of the Markov chain $\{r(k)\in \mathcal{X}\}_{k\in{\mathbb{N}_0}}$ for $k\in\{0,1,\ldots,10000\}$. Figure~\ref{fig:histogram} shows a histogram of $r(10000)$. We observe that it closely resembles the probability density function of the log-return at time $k\tau=2$.

Approximation of GBM with the Markov chain can be conceptually understood by $r(k) \approx \ln(s_{k\tau} / s_0)$, but in this section, we also assess the performance of the approximation $ s_0 e^{r(k)}\approx s_{k\tau}$. Specifically, Figure~\ref{fig:trajectories} shows 20 sample trajectories of  $s_0 e^{r(k)}$ as well as average of $N=10000$ sample trajectories. We note that the average trajectory closely matches $\mathbb{E}[s_{k\tau}]$. 

We remark that other time-scaling factors ($\tau$) can be used for approximation as long as the conditions of Corollary~\ref{corollary:GBM} are satisfied. For smaller values of $\tau$, obtaining approximations of the distribution of the log-return at a fixed time $k\tau$ take shorter, since trajectories of $r(k)$ need to be obtained for smaller $k$ values. To further reduce the simulation time, simulation of different sample trajectories of $r(k)$ can be carried out at different processing units in parallel. 


\begin{figure}[t]
    \centering
    \includegraphics[scale=0.5]{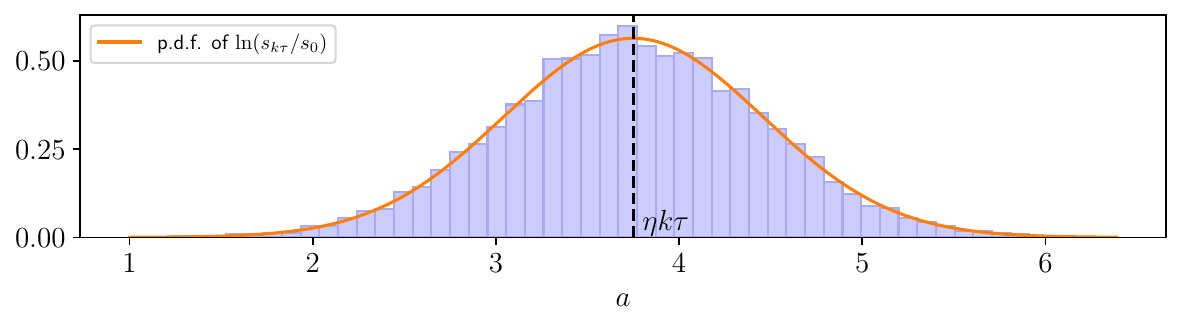} 
    \caption{Comparison of the probability density function of $\ln(s_{k\tau}/s_0)$ for $k\tau=2$ and histogram of $r(k)$ for $k=10000$. }\label{fig:histogram}
\end{figure}

\begin{figure}[t]
    \centering 
    \includegraphics[scale=0.5]{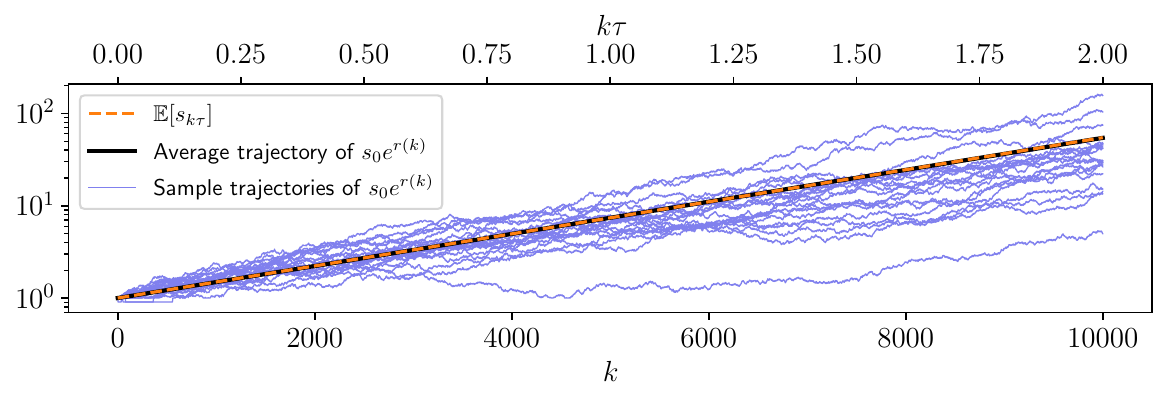} 
    \caption{Sample and average trajectories of $s_0 e^{r(k)}$ compared with $\mathbb{E}[s_{k\tau}]$.}\label{fig:trajectories}
\end{figure}

\begin{figure}[t]
    \centering 
    \includegraphics[scale=1.2]{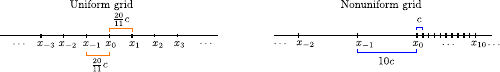} 
    \caption{Uniform and nonuniform grids with equal point density}\label{fig:wasserstein}
\end{figure}
\subsection{A Quantitative Comparison of Nonuniform and Uniform Grids}
Next, we compare whether the proposed Markov chain approximation works better on the nonuniform grid in \eqref{simul_grid} or a uniform grid of equal point density with points given as 
\begin{align}
    x_i= (20/11) \times ci, \quad i\in\mathbb Z,\label{uni_simul_grid}
\end{align}
where $c=0.01$ is the fixed constant from \eqref{simul_grid}. For this comparison, we investigate the Wasserstein $1$-distance from the distribution of the log-return. In particular, we use numerical integration to calculate
\begin{align*}
    W_1 (k)=\int_0^1\left|F_{1,k}^{-1}(q)-F_{2,k}^{-1}(q)\right|\mathrm{d}q,
\end{align*}
where $F_{1,k}^{-1}$ represents the empirical quantile function for $N=10000$ realizations of $r(k)$ and $F_{2,k}^{-1}$ represents the quantile function of the distribution of $\ln(s_{k\tau} / s_0)$ (i.e., normal distribution with mean $(\mu-\sigma^2/2)(k\tau)$ an variance $\sigma^2(k\tau)$). 

In Figure~\ref{fig:wasserstein}, we show the Wasserstein $1$-distance $W_1(k)$ obtained for $k\in\{1,\ldots,10000\}$ when we use a uniform grid and a nonuniform grid. We observe that for large values of $k$, Wasserstein $1$-distance obtained with the nonuniform grid is smaller than that obtained with the uniform grid. This is expected because when $t$ is large, $\ln(s_{k\tau} / s_0)$ is expected to take positive values due to $\eta$ being positive. Notice that in such a case, finer resolution on the positive side of the nonuniform grid allows better approximation. On the other hand, when $k$ is small, the distribution of the log-return $\ln(s_{k\tau} / s_0)$ is centered close to the $0$ value and spreads to both negative and positive values. For small $k$, the uniform grid being symmetric around $0$ allows a better representation, and thus a symmetric looking histogram (see top-left in Figure~\ref{fig:two-histograms}) and achieves a smaller Wasserstein $1$-distance value (see Figure \ref{fig:wasserstein}). For larger $k$, histograms are similar (see bottom plots in Figure~\ref{fig:two-histograms}) and the nonuniform grid achieves consistently smaller Wasserstein $1$-distance values.(Figure~\ref{fig:wasserstein}).

\begin{figure}[t]
    \centering 
    \includegraphics[scale=0.5]{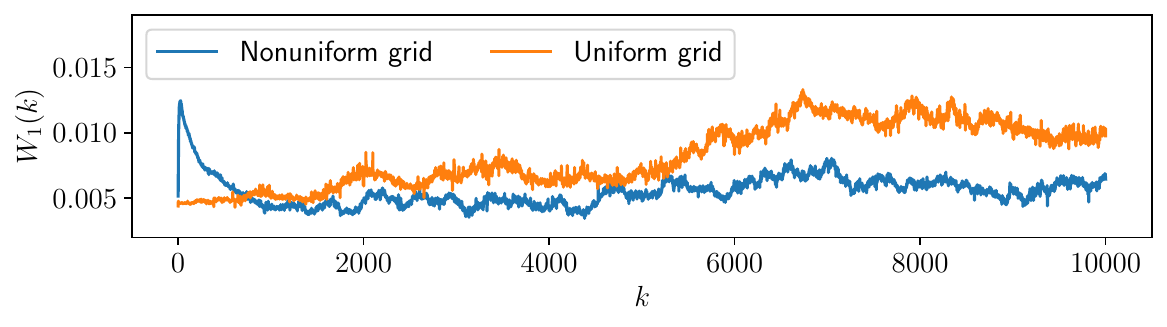}
    \centering
    \includegraphics[scale=0.5]{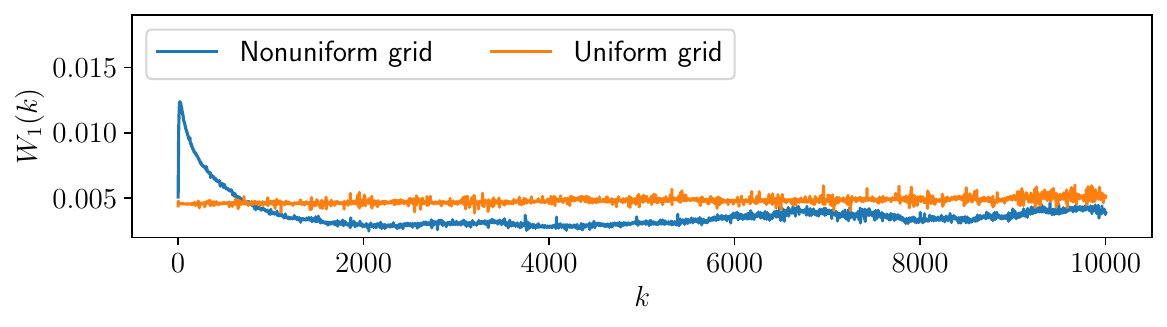}
    \caption{Comparison of Wasserstein $1$-distances obtained with uniform and nonuniform grids. Top: 10000 simulations, Bottom: 100000 simulations. }\label{fig:wasserstein}
\end{figure}

While we restricted our attention to constant drift and volatility coefficients, our method can also be used for the case where those coefficients are piecewise-constant functions of time. In that case, we can run the Markov chain simulation until the time there is a jump in the values of coefficients. Then, at the time of jump, we change the transition probabilities and start new simulation from the last location of the Markov chain on the grid.
\begin{figure}[t]
    \centering
    \begin{subfigure}[t]{0.45\linewidth}
        \centering
        \includegraphics[scale=0.45]{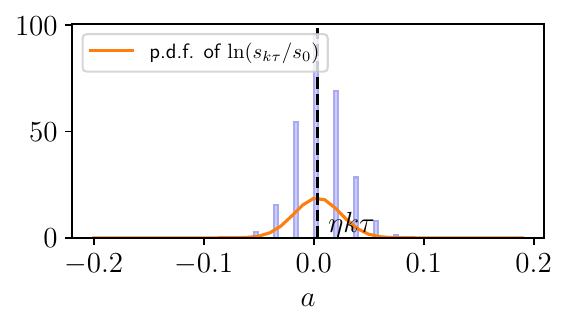} \\
        \includegraphics[scale=0.45]{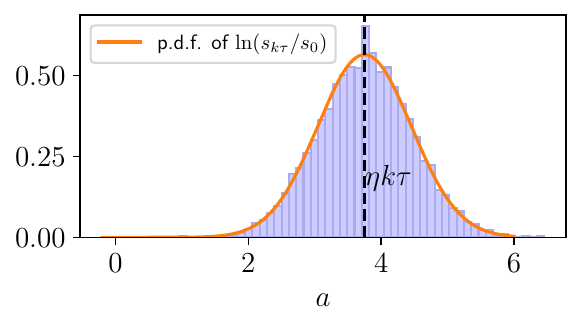}
    \end{subfigure}
    \begin{subfigure}[t]{0.45\linewidth}
        \centering
        \includegraphics[scale=0.45]{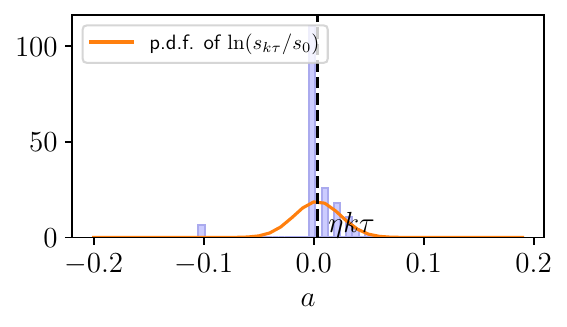}\\
        \includegraphics[scale=0.45]{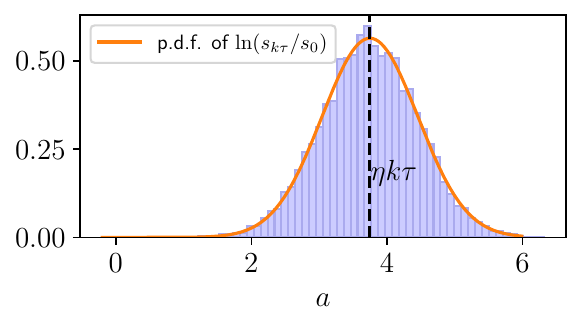}
    \end{subfigure}
    \caption{Probability density function of log-return $\ln(s_{k\tau} / s_0)$ compared to histograms of $r(k)$. Top-Left: Uniform grid for $k=10$, Top-Right: Nonuniform grid for $k=10$, Bottom-Left: Uniform grid for $k=10000$, Bottom-Right: Nonuniform grid for $k=10000$}
    \label{fig:two-histograms}
\end{figure}



\section{Conclusion}\label{sec:conclusion}
 In this study, we proposed a discrete-time Markov chain approach for approximating continuous time process with time-linear moments. In particular, we derived transition probabilities of a Markov chain on a nonuniform grid so as to guarantee that the expectation and the variance of the Markov chain matches those of a continuous process. We introduced a time-scaling factor to allow approximation at arbitrary times. We also discussed the approximation of geometric Brownian motion and heat diffusion based on our proposed approach. For future work, our goal is to extend our Markov chain approximation approach to other continuous process such as Levy processes including compound Poisson processes. We also aim to consider approximating regime-switching models.

\section*{Acknowledgements}
This work was supported by JSPS KAKENHI Grant No.~JP23K03913.

\bibliographystyle{plain}
\bibliography{references}

\end{document}